\documentclass[leqno, 12pt]{article}
\usepackage{amsmath}
\usepackage{amsthm}
\usepackage{amsfonts}
\usepackage{amssymb}

\makeatletter
 
  \@addtoreset{equation}{section}
 \makeatother

\newtheorem{thm}{Theorem}[section]
\newtheorem{lem}[thm]{Lemma}
\newtheorem{prop}[thm]{Proposition}

\theoremstyle{definition}

\theoremstyle{remark}

\begin{document}
\title{Unitary super perfect numbers
}
\author{Tomohiro Yamada}
\date{}
\maketitle

\begin{abstract}
We shall show that 9, 165 are all of the odd unitary super perfect numbers.
\end{abstract}

\section{Introduction}\label{intro}
We denote by $\sigma(N)$ the sum of divisors of $N$.
$N$ is called to be perfect if $\sigma(N)=2N$.
It is a well-known unsolved problem whether or not
an odd perfect number exists.  Interest to this problem
has produced many analogous notions.

D. Suryanarayana \cite{Sur} called $N$ to be
super perfect if $\sigma(\sigma(N))=2N$.  It is asked in this paper
and still unsolved whether there were odd super perfect numbers.

A special class of divisors is the class of unitary divisors defined
by Cohen \cite{Coh}.
A divisor $d$ of $n$ is called a unitary divisor if $(d, n/d)=1$.
Then we write $d\mid\mid n$.
We denote by $\sigma^*(N)$ the sum of unitary divisors of $N$.
Replacing $\sigma$ by $\sigma^*$, Subbarao and Warren \cite{SW}
introduced the notion of a unitary perfect number.  $N$
is called to be unitary perfect if $\sigma^*(N)=2N$.  They proved
that there are no odd unitary perfect numbers.  Moreover,
Subbarao \cite{Sub} conjectured that there are only finitely many
unitary perfect numbers.

Combining these two notions, Sitaramaiah and Subbarao \cite{SS}
studied unitary super perfect (USP) numbers, integers $N$ satisfying
$\sigma^*(\sigma^*(N))=2N$.  They found all unitary super perfect
numbers below $10^8$.  The first ones are $2, 9, 165, 238$.  Thus
there are both even and odd USPs.  They proved that
another odd USP must have at least four distinct prime factors
and conjectured that there are only finitely many odd USPs.

The purpose of this paper is to prove this conjecture.  Indeed,
we show that the known two USPs are all.
\begin{thm}\label{th11}
If $N$ is an odd USP, then $N=9$ or $N=165$.
\end{thm}

Our proof is completely elementary.  The key point of our proof
is the fact that if $N$ is an odd USP, then $\sigma^*(N)$ must
be of the form $2^{f_1} q^{f_2}$, where $q$ is an odd prime.
This yields that if $p^e$ is an unitary divisor of $N$, then $p^e+1$
must be of the form $2^a q^b$.  Moreover, elementary theory of
cyclotomic polynomials and quadratic residues gives that
$a\leq 2$ or $b=0$.  Hence $p^e$ belongs a to very thin set.
Using this fact, we deduces that $q$ must be small.
For each small primes $q$, we show that $\sigma^*(\sigma^*(N))/N<2$
and therefore $N$ cannot be an USP unless $N=9, 165$,
with the aid of the fact that $f_1, f_2$ must be fairly large.
We sometimes use facts already stated in \cite{SS} but we shall present proofs
of these facts when proofs are omitted in \cite{SS}.

Our method does not seem to work to find all odd super perfect numbers.
Since $\sigma(\sigma(N))=2N$ does not seem to imply that $\omega(\sigma(N))\leq 2$.
Even assuming that $\omega(\sigma(N))\leq 2$, the property of $\sigma$ that
$\sigma(p^e)/p^e>1+1/p$ prevents us from showing that $\sigma(\sigma(N))<2$.
Nevertheless, with the aid of a theory of exponential diophantine equations,
we can show that for any given $k$, there are only finitely many odd super
perfect numbers $N$ with $\omega(\sigma(N))\leq k$.

\section{Preliminary Lemmas}\label{lemmas}
Let us denote by $v_p(n)$ the solution $e$ of $p^e||n$.
For distinct primes $p$ and $q$, we denote by $o_q(p)$
the exponent of $p \mod q$ and we define $a_q(p)=v_q(p^d-1)$,
where $d=o_q(p)$.
Clearly $o_q(p)$ divides $q-1$  and $a_q(p)$ is a positive
integer. Now we quote some elementary properties of $v_q(\sigma(p^x))$.
Lemmas \ref{lm21} is well-known.  Lemma \ref{lm21} has been proved by
Zsigmondy\cite{Zsi} and rediscovered by many authors
such as Dickson\cite{Dic} and Kanold\cite{Kan}.  See also Theorem 6.4A.1
in \cite{Sha}.

\begin{lem}\label{lm21}
If $a>b\ge 1$ are coprime integers, then $a^n-b^n$ has a prime factor
which does not divide $a^m-b^m$ for any $m<n$, unless $(a, b, n)=(2, 1, 6)$
or $a-b=n=1$, or $n=2$ and $a+b$ is a power of $2$.
\end{lem}

By Lemma \ref{lm21}, we obtain the following lemmas.

\begin{lem}\label{lm22}
Let $p$, $q$ be odd primes and $e$ be a positive integer.
If $p^e+1=2^aq^b$ for some integers $a$ and $b$, then
one of the following holds:

a)$e=1$.

b)$e$ is even and $q\equiv 1\pmod{2e}$.

c)$p$ is a Mersenne prime and $q\equiv 1\pmod{2e}$.
\end{lem}

\begin{proof}
We first show that if a) does not hold, then either b) or c) must hold.
Since $(p, e)\neq (2, 3)$ and $e\neq 1$, it follows from Lemma \ref{lm21}
that $p^{2e}-1$ has a prime factor $r$ which does not divide $p^m-1$ for any $m<2e$.
Since the order of $p\pmod{r}$ is $2e$, $r\equiv 1\pmod{2e}$.
Since $r$ is odd and does not divide $p^e-1$, we $r$ divides $p^e+1$ and
therefore $q=r$.

If $e$ is even, then b) holds.  Assume that $e$ is odd.  If $p+1$ has an odd prime factor,
then this cannot be equal to $q$ and must be a prime factor of $p^e+1=2^aq^b$,
which is contradiction.  Thus $p$ is a Mersenne prime and c) follows.
\end{proof}

\begin{lem}\label{lm23}
Let $p$ be an odd prime and $e$ be a positive integer.
If $p^e+1=2^a3^b$ for some integers $a$ and $b$, then $e=1$.
\end{lem}
\begin{proof}
By Lemma \ref{lm22}, $e=1$ or $3\equiv 1\pmod{2e}$.  The latter is equivalent to $e=1$.
\end{proof}

\begin{lem}\label{lm24}
Let $p$, be an odd prime and $e, x$ be positive integers.
If $p^e+1=2^x$, then $e=1$.
\end{lem}
\begin{proof}
If $e>1$, then by Lemma \ref{lm21}, $p^{2e}-1$ has a prime factor which does not
divide $p^m-1$ for any $m<2e$.  This prime factor must be odd and divide
$p^e+1$, which violates the condition $p^e+1=2^x$.
\end{proof}

\begin{lem}\label{lm25}
Let $p$, be an odd prime and $e, x$ be positive integers.
If $2^x+1=3^e$, then $(e, x)=(1, 1)$ or $(2, 3)$.
\end{lem}
\begin{proof}
We apply Lemma \ref{lm21} with $(a, b, n)=(3, 1, e)$.  If $e>2$,
then $3^e-1$ has a prime factor which does not divide $3-1=2$.
\end{proof}

\begin{lem}\label{lm26}
If a prime $p$ divides $2^a+1$ for some integer $a$, then $p$ is congruent to
$1, 3$ or $5\pmod{8}$.
\end{lem}
\begin{proof}
If $a$ is even, then it is well known that $p\equiv 1\pmod{4}$.  If $a$ is odd,
then $p$ divides $2x^2+1$ with $x=2^{(a-1)/2}$.  We have $(-2/p)=1$ and therefore
$p\equiv 1$ or $3\pmod{8}$.
\end{proof}

\begin{lem}\label{lm27}
Let $p$ and $q$ be odd primes and $b$ be a positive integer.
If a prime $p$ divides $q^b+1$ and $4$ does not divide $q^b+1$, then $4q$ does not
divide $p+1$.
\end{lem}
\begin{proof}
If $b$ is even, then $p\equiv 1\pmod{4}$ and clearly $4q$ does not divide $p+1$.

If $b$ is odd, then we have $(-q/p)=1$ and $q\equiv 1\pmod{4}$.  Assume that
$q$ divides $p+1$.  Since $q\equiv 1\pmod{4}$, we have, by the reciprocity law,
$(-q/p)=(-1/p)(q/p)=(-1/p)(p/q)=(-1/p)(-1/q)=(-1/p)$.  Thus $(-1/p)=1$ and
$p\equiv 1\pmod{4}$ and therefore $4$ does not divide $p+1$.
\end{proof}

\section{Basic properties of odd USPs}
In this section, we shall show some basic properties of odd USPs.  

We write $N=p_1^{e_1} p_2^{e_2}\ldots p_k^{e_k}$, where $p_1, p_2, \ldots, p_k$ are
distinct primes.  Moreover, we denote by $C$ the constant
\begin{equation}
\prod_{p, 2^p-1\mathrm{\ is\ prime}}\frac{2^p}{2^p-1}<1.6131008.
\end{equation}
This upper bound follows from the following estimate:
\begin{equation}\label{eq31}
\begin{split}
\prod_{p, 2^p-1\mathrm{\ is\ prime}}\frac{2^p}{2^p-1}&<\frac{4}{3}\cdot\left(\prod_{n\ge 3, n\mathrm{\ is\ odd}}\frac{2^n}{2^n-1}\right)\\
&<\frac{4}{3}\cdot\exp\left(\sum_{n\ge 3, n\mathrm{\ is\ odd}}\frac{1}{2^n-1}\right)\\
&<\frac{4}{3}\cdot\exp\left(\frac{1}{7}\sum_{n\ge 0}\frac{1}{4^n}\right)\\
&=\frac{4}{3}\cdot\exp\left(\frac{4}{21}\right)=1.631007\cdots.\\
\end{split}
\end{equation}

\begin{lem}\label{lm31}
If $N$ is an odd USP, then $\sigma^*(N)=2^{f_1}q^{f_2}$ for some odd prime $q$ and
positive integers $f_1, f_2$.  Moreover, $q^{f_2}+1$ is not divisible by $4$.
\end{lem}

\begin{proof}
Since $N$ is odd, $\sigma^*(N)$ must be even.
Moreover, since $\sigma^*(\sigma^*(N))=2N$ with $N$ odd, $\sigma^*(N)$ has
exactly one odd prime factor.  Hence $\sigma^*(N)=2^{f_1}q^{f_2}$ for some
odd prime $q$ and positive integers $f_1, f_2$.  Since $\sigma^*(q^{f_2})=q^{f_2}+1$
divides $\sigma^*(\sigma^*(N))=2N$, $4$ does not divide $q^{f_2}+1$.
\end{proof}

Henceforth, we let $N\neq 9, 165$ be an odd USP and
write $\sigma^*(N)=2^{f_1}q^{f_2}$ as allowed by Lemma \ref{lm31}.

\begin{lem}\label{lm32}
Unless $p_i$ is a Mersenne prime and $e_i$ is odd, we have $p_i^{e_i}=2^{a_i}q^{b_i}-1$
for some positive integers $a_i$ and $b_i$ with $a_i\leq 2$.
Moreover, $f_1=\sum_{i=1}^{k} a_i$ and $f_2=\sum_{i=1}^{k} b_i$.
\end{lem}

\begin{proof}
Since $\sigma^*(p_i^{e_i}+1)$ divides $\sigma^*(N)=2^{f_1}q^{f_2}$,
we can write $p_i^{e_i}+1=2^{a_i}q^{b_i}$ with some nonnegative integers
$a_i$ and $b_i$.  Since $p_i$ is odd and non-Mersenne, $a_i$ and $b_i$
are positive by Lemma \ref{lm24}.

If $e_i$ is even, then $p_i^{e_i}+1\equiv 2\pmod{4}$.  Hence $a_i=1$.

Assume that $p_i$ is not a Mersenne prime and $e_i$ is odd.
By Lemma \ref{lm22}, we have $e_i=1$ and therefore
$p_i=p_i^{e_i}=2^{a_i}q^{b_i}-1$.  By Lemma \ref{lm26} and \ref{lm27},
we have $a_i\leq 2$ since $q^{f_2}+1$ is not divisible by $4$.  This shows
$a_i\leq 2$.  The latter part of the lemma immediately follows from
$2^{f_1}q^{f_2}=\sigma^*(N)=\prod (p_i^{e_i}+1)$.
\end{proof}

\begin{lem}\label{lm34}
$\omega(N)\geq 3$.
\end{lem}

\begin{proof}
First we assume that $N=p_1^{e_1}$.  Since we have
$\sigma^*(N)/N=1+1/N$ and $\sigma^*(\sigma^*(N))/\sigma^*(N)\leq (1+1/2)(1+2/N)$
by Lemma \ref{lm31},
we have $N\leq 9$.  We can easily confirm that $N=9$ is the sole odd USP with
$N\leq 9$.

Next we assume that $N=p_1^{e_1}p_2^{e_2}$.  Since we have
$\sigma^*(N)/N\leq (1+1/3)(1+3/N)$ and $\sigma^*(\sigma^*(N))/\sigma^*(N)\leq (1+1/4)(1+4/N)$,
we have $N<37$.  We can easily confirm that there is no odd USP $N$ with
$N<37$ and $\omega(N)=2$.

Another proof of impossibility of $\omega(N)=1$ unless $N=2, 9$ (whether $N$ is even or odd)
can be found in \cite[Theorem 3.2]{SS} and impossibility of $\omega(N)=2$ (again, $N$ may be even)
is stated in \cite[Theorem 3.3]{SS} with their proof presented only in the case $N$ is even.
\end{proof}

\section{$q$ cannot be $3$}

In this section, we show that $q\neq 3$.
There are two cases: the case $3\mid N$ and the case $3\nmid N$.

\begin{prop}\label{pr41}
If $3\nmid N$ and $3\mid \sigma^*(N)$, then $f_1$ and $f_2$ are even, $p_i$ has
the form $2\cdot 3^{b_i}-1$ with positive integers $b_i$.
\end{prop}
\begin{proof}
We have $e_i=1$ by Lemma \ref{lm23}.
Thus any $p_i$ must be of the form $2^{a_i}\cdot 3^{b_i}-1$ with
nonnegative integers $a_i, b_i$.
Since $3^{f_2}+1$ is not divisible by $4$, $f_2$ must be even.
Since $3$ does not divide $2^{f_1}+1$, $f_1$ must also be even.
By Lemma \ref{lm26}, any prime factor of $N$ is congruent to $1\pmod{4}$
and therefore $a_i$ must be odd.
By Lemma \ref{lm32}, we have $a_i=1$.
\end{proof}

Hence we have $p_i\in\{5, 17, 53, 4373, \ldots\}$.

\begin{lem}\label{lm42}
If $3\mid \sigma^*(N)$, then $3\mid N$.
\end{lem}
\begin{proof}
Suppose $3\mid \sigma^*(N)$ and $3\nmid N$.
By Proposition \ref{pr41}, we have
\begin{equation}
\frac{\sigma^*(N)}{N}\leq\frac{6}{5}\cdot\frac{18}{17}\cdot\frac{54}{53}\cdot\left(\prod_{i=7}^{\infty}\frac{2\cdot 3^i}{2\cdot 3^i-1}\right).
\end{equation}
Since
\begin{equation}
\prod_{i=7}^{\infty}\frac{2\cdot 3^i}{2\cdot 3^i-1}\leq\exp\sum_{i=7}^{\infty}\frac{1}{2\cdot 3^i-1}\leq\linebreak[0]\exp\left(\frac{1}{2\cdot 3^7-1}\sum_{i=0}^{\infty}3^{-i}\right),
\end{equation}
we have
\begin{equation}\label{eq41}
\frac{\sigma^*(N)}{N}<\frac{6}{5}\cdot\frac{18}{17}\cdot\frac{54}{53}\cdot\exp\left(\frac{3}{8744}\right).
\end{equation}

Since $k\geq 3$ by Lemma \ref{lm34}, we have $f_1=k\geq 3$ and $f_2\geq 3+2+1=6$.
Thus we obtain
\begin{equation}\label{eq42}
\frac{\sigma^*(\sigma^*(N))}{\sigma^*(N)}\leq\frac{9}{8}\cdot\frac{730}{729}.
\end{equation}
Multiplying (\ref{eq41}) and (\ref{eq42}), we obtain
\begin{equation}
2=\frac{\sigma^*(\sigma^*(N))}{N}<\frac{9}{8}\cdot\frac{730}{729}\cdot
\frac{6}{5}\cdot\frac{18}{17}\cdot\frac{54}{53}\cdot\exp\left(\frac{3}{8744}\right)=1.4588\cdots<2,
\end{equation}
which is contradiction.
\end{proof}

\begin{lem}\label{lm43}
It is impossible that $3\mid N$ and $3\mid \sigma^*(N)$.
\end{lem}
\begin{proof}
Suppose $3\mid N$ and $3\mid \sigma^*(N)$.  We have $e_i=1$ by Lemma \ref{lm23}.
By Lemma \ref{lm26}, $2^{a_i}+1$ is divisible by no Mersenne prime other than 3.
Since $3^{b_i}+1$ cannot be divisible by 4, $b_i$ must be odd and therefore
$3^{b_i}+1$ is divisible by no Mersenne prime.  Hence it follows from Lemma \ref{lm32}
that any $p_i$ must be of the form $2^{a_i}\cdot 3^{b_i}-1$, where $a_i\leq 2$ and $b_i$
are positive integers.
Hence $p_i\in\{5, 11, 17, 53, 107, 971, 4373, \ldots\}$.

Thus we obtain
\begin{equation}
\frac{\sigma^*(N)}{N}\leq\frac{4}{3}\cdot\frac{6}{5}\cdot\frac{12}{11}\cdot\frac{18}{17}\cdot\frac{54}{53}\cdot\frac{108}{107}\cdot\left(\prod_{i=7}^{\infty}\frac{2\cdot 3^i}{2\cdot 3^i-1}\right)
\cdot\left(\prod_{i=5}^{\infty}\frac{4\cdot 3^i}{4\cdot 3^i-1}\right).
\end{equation}
As in the proof of the previous lemma, substituting the inequality
\begin{equation}
\prod_{i=5}^{\infty}\frac{4\cdot 3^i}{4\cdot 3^i-1}\leq\exp\left(\frac{1}{4\cdot 3^5-1}\sum_{i=0}^{\infty}3^{-i}\right)
\end{equation}
we have
\begin{equation}\label{eq43}
\frac{\sigma^*(N)}{N}\leq\frac{4}{3}\cdot\frac{6}{5}\cdot\frac{12}{11}\cdot\frac{18}{17}\cdot\frac{54}{53}\cdot\frac{108}{107}\cdot\exp\left(\frac{3}{8744}+\frac{3}{1942}\right).
\end{equation}

Since $k\geq 46$ by \cite[Theorem 3.4]{SS}, we have
\begin{equation}\label{eq44}
\frac{\sigma^*(\sigma^*(N))}{\sigma^*(N)}\leq\frac{2^{46}+1}{2^{46}}\cdot\frac{3^{45}+1}{3^{45}}.
\end{equation}
Multiplying (\ref{eq43}) and (\ref{eq44}), we obtain
\begin{equation}
\begin{split}
2&=\frac{\sigma^*(\sigma^*(N))}{N}\\&\leq\frac{2^{46}+1}{2^{46}}\cdot\frac{3^{45}+1}{3^{45}}\cdot\frac{4}{3}\cdot\frac{6}{5}\cdot\frac{12}{11}\cdot\frac{18}{17}\cdot\frac{54}{53}\cdot\frac{108}{107}\cdot\exp\left(\frac{3}{8744}+\frac{3}{1942}\right)\\&\leq 1.9041\cdots<2,
\end{split}
\end{equation}
which is contradiction.
\end{proof}

It immediately follows from these two lemmas that $q\neq 3$.

\section{The remaining part}

The remaining case is the case $3\nmid \sigma^*(N)$, i.e., $q\neq 3$.

\begin{lem}\label{lm51}
Suppose $p_i$ is not a Mersenne prime. Then $p_i^{e_i}$ has the form $2^{a_i}\cdot q^{b_i}-1$
with positive integers $a_i\leq 2$ and $b_i$.  Moreover, for any integer $b$,
at most one of the pairs $(1, b)$ and $(2, b)$ appear in $(a_i, b_i)$'s.
\end{lem}
\begin{proof}
The former part follows from Lemma \ref{lm32}.  Since $q\neq 3$, $3$ divides
at least one of $2\cdot q^{b}-1$ and $4\cdot q^{b}-1$.  If both pairs
$(a_i, b_i)=(1, b)$ and $(a_j, b_j)=(2, b)$ appear, then at least one of $p_i^{e_i}$
and $p_j^{e_j}$ must be a power of three, which violates the condition
that $p_i$ and $p_j$ are not Mersenne.
\end{proof}

\begin{lem}\label{lm52}
$q\leq 13$.  Furthermore, provided $f_2\geq 2$, we have $q=5$ or $q=7$.
\end{lem}
\begin{proof}
By Lemma \ref{lm51}, we have
\begin{equation}
\frac{\sigma^*(N)}{N}\leq C\cdot\left(\prod_{a=1}^{\infty}\frac{2\cdot q^a}{2\cdot q^a-1}\right).
\end{equation}
Since
$\prod_{a=1}^{\infty}2\cdot q^a/(2\cdot q^a-1)\leq\exp\left(q/\left\{(q-1)(2q-1)\right\}\right)$,
we have
\begin{equation}
\frac{\sigma^*(N)}{N}\leq C\cdot\exp\left(\frac{q}{(q-1)(2q-1)}\right).
\end{equation}

By Lemma \ref{lm34}, we have
\begin{equation}
\frac{\sigma^*(\sigma^*(N))}{\sigma^*(N)}\leq\frac{2^{f_1}+1}{2^{f_1}}\cdot\frac{q^{f_2}+1}{q^{f_2}}\leq\frac{2^3+1}{2^3}\cdot\frac{q^{f_2}+1}{q^{f_2}}.
\end{equation}

Combining these inequalities, we obtain
\begin{equation}
2\leq\frac{\sigma^*(\sigma^*(N))}{N}\leq\frac{2^3+1}{2^3}\cdot C\cdot\frac{q^{f_2}+1}{q^{f_2}}\cdot\exp\left(\frac{q}{(q-1)(2q-1)}\right).
\end{equation}

Hence
\begin{equation}
\frac{q^{f_2}+1}{q^{f_2}}\cdot\exp\left(\frac{q}{(q-1)(2q-1)}\right)\geq\frac{16}{9C}\geq 1.102087.
\end{equation}
This yields $q\leq 13$.  If $f_2\geq 2$, then this inequality yields $q\leq 7$.
\end{proof}

\begin{thm}\label{th53}
$q\neq 5$.
\end{thm}
\begin{proof}
Suppose that $q=5$.  Then we have $p_i^{e_i}=2\cdot 5^{b_i}-1$ or
$p_i^{e_i}=4\cdot 5^{b_i}-1$ or $p_i$ is Mersenne.
Hence $p_i^{e_i}\in\{ 19, 499, 7812499, \ldots, 9, 49, 1249, \ldots,\linebreak[0] 3, 7, 31, 127, 8191, \ldots\}$.
We note that $9=3^2$ and $49=7^2$.

Let us assume that $19\mid N$.  Then $f_1\equiv 9\pmod{18}$ and hence $3^3\mid N$.
By (\ref{eq31}), we have
\begin{equation}
\frac{\sigma^*(N)}{N}\leq \frac{3}{4}\cdot\frac{28}{27}\cdot C\cdot\exp\left(\frac{5}{36}\right).
\end{equation}
Since $f_1\geq 9$, we have
\begin{equation}
\frac{\sigma^*(\sigma^*(N))}{N}\leq\frac{2^9+1}{2^9}\cdot\frac{6}{5}\cdot\frac{7}{9}\cdot C\cdot\exp\left(\frac{5}{36}\right)=1.7332\cdots <2,
\end{equation}
which is contradiction.  Thus $19$ cannot divide $N$.  From this we deduce that
if $p_i^{e_i}=2\cdot 5^{b_i}-1$ or $p_i^{e_i}=4\cdot 5^{b_i}-1$, then $b_i\geq 3$.

It is impossible that $7\mid N$ since $7$ does not divide $2^x+1$ or $5^x+1$
for any integer $x$.

Hence, by Lemma \ref{lm34} we have
\begin{equation}
\frac{\sigma^*(\sigma^*(N))}{N}\leq \frac{7}{8}\cdot C\cdot\exp\left(\frac{5}{4}\cdot\frac{250}{249}\right)\cdot\frac{6}{5}\cdot\frac{9}{8}<1.9150\cdots <2.
\end{equation}

So that, we cannot have $q=5$.
\end{proof}

\begin{thm}\label{th54}
$q\neq 7, 11, 13$.
\end{thm}
\begin{proof}
Suppose $q=7$.  Observing that $4\cdot 7^b-1$ is divisible by $3$, we deduce
from Lemma \ref{lm32} that, for any $i$, $p_i$ is a Mersenne prime
or $p_i^{e_i}=2\cdot 7^{b_i}-1$.  By Lemma \ref{lm26}, $(2^{f_1}+1)(7^{f_2}+1)$
is not divisible by $7$.  Hence
\begin{equation}\label{eq541}
\begin{split}
\frac{\sigma^*(N)}{N}&\leq \frac{4}{3}\cdot\left(\prod_{i=2}^{\infty}\frac{2^{2i+1}}{2^{2i+1}-1}\right)
\cdot\left(\prod_{i=1}^{\infty}\frac{2\cdot 7^i}{2\cdot 7^i-1}\right)\\
&\leq\frac{4}{3}\cdot\exp\left(\frac{1}{31}\cdot\frac{4}{3}+\frac{1}{13}\cdot\frac{8}{7}\right).
\end{split}
\end{equation}

By Lemma \ref{lm34}, we have $k\geq 3$.  We deduce from Lemma \ref{lm32} that
we can take an integer $s$ with $1\leq s\leq 3$ for which the following statement holds:
there is at least $3-s$ indices $i$ such that $p_i$ is a Mersenne prime and $e_i$ is odd,
and there is at least $s$ indices $i$ such that $p_i^{e_i}=2\cdot 7^{b_i}-1$.
If $s=1$, then $f_1\geq 6$ and $f_2\geq 1$.  If $s=2$, then $f_1\geq 4$ and $f_2\geq 3$.
If $s=3$, then $f_1\geq 3$ and $f_2\geq 6$.
\begin{equation}\label{eq542}
\begin{split}
\frac{\sigma^*(\sigma^*(N))}{\sigma^*(N)}&\leq\max\left\{\frac{2^6+1}{2^6}\cdot\frac{8}{7}\cdot\frac{2^4+1}{2^4}\cdot\frac{7^3+1}{7^3}, \frac{2^3+1}{2^3}\cdot\frac{7^6+1}{7^6}\right\}\\
&\leq\frac{65}{56}.
\end{split}
\end{equation}
Combining two inequalities (\ref{eq541}) and (\ref{eq542}), we have
\begin{equation}
\frac{\sigma^*(\sigma^*(N))}{N}\leq\frac{65}{56}\cdot\frac{4}{3}\cdot\exp\left(\frac{1}{31}\cdot\frac{4}{3}+\frac{1}{13}\cdot\frac{8}{7}\right)=1.7604\cdots<2,
\end{equation}
which is contradiction.

Suppose $q=11$.  Observing that $2\cdot 11^{2b+1}-1$ and $4\cdot 11^{2b}-1$ is divisible by $3$,
we deduce from Lemma \ref{lm32} that, for any $i$, $p_i$ is a Mersenne prime
or $p_i^{e_i}=2^{a_i}\cdot 7^{b_i}-1$ with $a_i+b_i$ odd.
\begin{equation}
\begin{split}
\frac{\sigma^*(N)}{N}&\leq\frac{4}{3}\cdot\left(\prod_{i=2}^{\infty}\frac{2^{2i+1}}{2^{2i+1}-1}\right)
\cdot\left(\prod_{i=1}^{\infty}\frac{2\cdot 11^i}{2\cdot 11^i-1}\right)\\
&\leq\frac{4}{3}\cdot\exp\left(\frac{1}{31}\cdot\frac{4}{3}+\frac{1}{21}\cdot\frac{12}{11}\right).
\end{split}
\end{equation}

In a similar way to derive (\ref{eq542}), we obtain
\begin{equation}
\frac{\sigma^*(\sigma^*(N))}{\sigma^*(N)}\leq\frac{2^3+1}{2^3}\cdot\frac{11^6+1}{11^6}.
\end{equation}

Combining these inequalities, we have
\begin{equation}
\begin{split}
2=\frac{\sigma^*(\sigma^*(N))}{N}&\leq\frac{2^3+1}{2^3}\cdot\frac{11^6+1}{11^6}\cdot\frac{4}{3}
\cdot\frac{8}{7}\cdot\exp\left(\frac{1}{31}\cdot\frac{4}{3}+\frac{1}{21}\cdot\frac{12}{11}\right)\\
&\leq 1.8850\cdots<2,
\end{split}
\end{equation}
which is contradiction.

Suppose $q=13$.  $3\mid N$ and $3\nmid (q^{f_2}+1)$ since $q=13\equiv 1\pmod{3}$.
Hence $f_1$ must be odd.  Moreover, $f_2=1$ by Lemma \ref{lm52}.  Hence
$\sigma^*(N)=2^{f_1}\cdot 13$ and $N=7(2^{f_1}+1)$.
There is exactly one index $j$ such that $p_j^{e_j}$ is of the form
$2^{a}13^{b}-1$ for some positive integers $a, b$.  By Lemma \ref{lm32}, we have $a\leq 2$.  Moreover, we have $b=1$ since $b\leq f_2=1$.  Hence $p_j^{e_j}=25=5^2$.
Since $13^{f_2}+1=2\cdot 7$, $2^{f_1}+1$ must be divisible by $5$.  But this is
impossible since $f_1$ is odd.
\end{proof}

Now Theorem \ref{th11} is clear.  By Lemma \ref{lm52}, $q$ must be one of
$3, 5, 7, 11, 13$.  In the previous section, it is shown that $q\neq 3$.
Theorem \ref{th53} shows that $q\neq 5$.  Theorem \ref{th54} eliminates
the remaining possibilities.

{}
\vskip 12pt

{\small Tomohiro Yamada}\\
{\small Department of Mathematics\\Faculty of Science\\Kyoto University\\Kyoto, 606-8502\\Japan}\\
{\small e-mail: \protect\normalfont\ttfamily{tyamada@math.kyoto-u.ac.jp}}
\end{document}